\documentclass[12pt]{article}
\usepackage{amssymb,amsfonts,amsmath, psfrag,eepic,colordvi,graphicx,epsfig}
\usepackage{amssymb,latexsym,graphics,array}
\usepackage{MnSymbol}
\usepackage[enableskew]{youngtab}
\usepackage{float,enumerate,enumitem}
\usepackage{tikz,ifthen}
\usepackage{pgfplots}
\usepackage{hyperref}
\usepackage{booktabs}
\usepackage{mathtools}
\usetikzlibrary{math}
\usepackage{physics}
\usepackage{amsthm}
\usepackage{booktabs}  
\usepackage{cleveref}

\hypersetup{colorlinks=true}
\allowdisplaybreaks[4]

\topskip  
\numberwithin{equation}{section}
\newtheorem{theorem}{Theorem}[section]
\newtheorem{proposition}[theorem]{Proposition}

\newtheorem{lemma}[theorem]{Lemma}

\theoremstyle{definition}

\newtheorem{remark}[theorem]{Remark}

\begin{document}
	\parskip 6pt
	
	\pagenumbering{arabic}
	\def\sof{\hfill\rule{2mm}{2mm}}
	\def\ls{\leq}
	\def\gs{\geq}
	\def\SS{\mathcal S}
	\def\qq{{\bold q}}
	\def\MM{\mathcal M}
	\def\TT{\mathcal T}
	\def\EE{\mathcal E}
	\def\lsp{\mbox{lsp}}
	\def\rsp{\mbox{rsp}}
	\def\pf{\noindent {\it Proof.} }
	\def\mp{\mbox{pyramid}}
	\def\mb{\mbox{block}}
	\def\mc{\mbox{cross}}
	\def\qed{\hfill \rule{4pt}{7pt}}
	\def\block{\hfill \rule{5pt}{5pt}}
	\def\Asctop{\mathrm{Asctop}}
	\def\Ascbot{\mathrm{Ascbot}}
	\def\Desbot{\mathrm{Desbot}}
	\def\Destop{\mathrm{Destop}}
	\def\asctop{\mathrm{asctop}}
	\def\ascbot{\mathrm{ascbot}}
	\def\desbot{\mathrm{desbot}}
	\def\destop{\mathrm{destop}}
	\def\A{\mathcal{A}}
	\def\MA{\hat{\mathcal{A}}}
	\def\RA{\hat{\mathcal{B}}}
	\def\CC{\hat{\mathcal{C}}}
	\def\DD{\hat{\mathcal{D}}}
	\def\nub{\mathrm{Nub}}
	\def\Cay{\text{Cay}}
	\def\Dec{\mathrm{Dec}}
	\def\hh{\text{hat}}
\def\ha{\eta}
	\def\Add{\mathrm{Add}}
	\def\std{\mathrm{std}}
	\def\T{\mathcal{T}}

	\def\lr#1{\multicolumn{1}{|@{\hspace{.6ex}}c@{\hspace{.6ex}}|}{\raisebox{-.3ex}{$#1$}}}
	\def\red{\textcolor{red}}

	\begin{center}
{\Large \bf Arrow-Wilf equivalences and enumerative results for short arrow patterns}
	\end{center}
	
	\begin{center}		
			{\small  Robin D.P. Zhou$^{*}$,\footnote{$^*$Corresponding author.}  \footnote{{\em E-mail address:} dapao2012@163.com. } Xinyang Yu}
		
		College of Mathematics Physics and Information\\
		Shaoxing University\\
		Shaoxing 312000, P.R. China\\

	\end{center}
	
\noindent {\bf Abstract.} 	
Arrow patterns, introduced by Berman and Tenner, provide a unified framework for studying permutation classes where both one-line and cycle structure constraints are present. In this paper, we continue the systematic study of arrow pattern avoidance initiated by Archer and Laudone. We establish several structural results, including a key lemma that translates arrow patterns into vincular patterns under certain conditions, and derive a series of arrow-Wilf equivalences arising from reversal, complementation, and insertion operations. We also resolve the two cases $(12;3\to 3)$ and $(21;3\to 3)$ left open by Archer and Laudone, and enumerate the arrow patterns of the form $(\nu; b\to c)$ of size $3$ with $\nu \in \{31, 23, 32\}$ and $b,c\in [3]$, providing explicit formulas connecting the results to Bell numbers, Bessel numbers, Catalan numbers, and derangement numbers. Together with earlier work of Archer and Laudone, this leaves only $(32;1\to 3)$ unresolved for $|\nu|\le 2$, which we pose as an open problem.
	


	
	\section{Introduction}

Let $[n] = \{1,2,\ldots,n\}$ and let $\mathcal{S}_n$ denote the set of permutations of $[n]$.
A permutation $\pi \in \mathcal{S}_n$ can be written in one-line notation as $\pi = \pi_1\pi_2\cdots\pi_n$, where $\pi_i = \pi(i)$.
It can also be expressed in cycle notation.
Throughout this paper, we use the standard cycle representation, where each cycle is written with its largest element first, and the cycles are ordered by increasing largest elements.
For instance, the one-line permutation $\pi = 6257341$ has standard cycle form $(2)(53)(7164)$.

There is a classical bijection $\theta: \mathcal{S}_n \rightarrow \mathcal{S}_n$, often called \emph{Foata’s first
	fundamental transformation}~\cite{Lothaire}, which connects the one-line notation of a permutation with its standard cycle form. 
Given a permutation written in standard cycle notation, the map $\theta$ removes the parentheses to produce a permutation in one-line notation. 
Conversely, the inverse map $\theta^{-1}$ recovers the standard cycle form from the one-line notation by inserting a left parenthesis before each left-to-right maximum and a right parenthesis immediately before the next left-to-right maximum (or at the end of the permutation). 
For example, continuing with the permutation from the previous paragraph, we have $\theta(6257341) = 2537164$. 
Conversely, if $\pi = 2537164$, then $\hat{\pi} = \theta^{-1}(\pi) = (2)(53)(7164) = 6257341$.
We denote $\hat{\pi} = \theta^{-1}(\pi)$ throughout this paper.

Given a sequence of distinct positive integers $w = w_1w_2\cdots w_k$, its \emph{reduction}, denoted $\operatorname{red}(w)$, is the permutation in $\mathcal{S}_k$ obtained by replacing the $i$-th smallest entry of $w$ with $i$. 
For example, $\operatorname{red}(3715) = 2413$.
A permutation $\pi = \pi_1\pi_2\cdots\pi_n \in \mathcal{S}_n$ is said to \emph{contain} a pattern $\sigma \in \mathcal{S}_k$ if there exist indices $i_1 < i_2 < \cdots < i_k$ such that $\operatorname{red}(\pi_{i_1}\pi_{i_2}\cdots\pi_{i_k}) = \sigma$. 
If no such subsequence exists, we say that $\pi$ \emph{avoids} $\sigma$, or that $\pi$ is $\sigma$-avoiding. 
For instance, the permutation $\pi = 6257341$ contains two occurrences of the pattern $123$, namely the subsequences $257$ and $234$, but it avoids $1234$ since it has no increasing subsequence of length $4$. 
For a set of patterns $\Sigma$, we write $\mathcal{S}_n(\Sigma)$ for the set of permutations in $\mathcal{S}_n$ that avoid every pattern in $\Sigma$. 
When $\Sigma$ consists of a single pattern $\sigma$, we simply write $\mathcal{S}_n(\sigma)$.

The study of pattern avoidance in permutations traces back to the seminal work of Knuth~\cite{Knuth} on stack-sortable permutations, in which he proved that permutations in $\mathcal{S}_n$ avoiding the pattern $231$ are counted by the $n$-th Catalan number $C_n = \frac{1}{n+1}\binom{2n}{n}$. Since then, the subject has developed into a major branch of enumerative combinatorics; we refer the reader to the books of Bóna~\cite{Bona} and Kitaev~\cite{Kitaev} for comprehensive overviews.


A natural generalization of classical pattern avoidance is the notion of vincular patterns, introduced by Babson and Steingrímsson~\cite{Babson}. 
In a \emph{vincular pattern}, an overline connecting adjacent entries indicates that those entries must appear consecutively in any occurrence.
More precisely, an occurrence of a vincular pattern in a permutation is a subsequence that is order-isomorphic to the underlying classical pattern, with the additional requirement that entries corresponding to adjacent positions joined by an overline appear consecutively in the permutation. 
For example, consider the vincular pattern $\overline{21}3$, where the entries corresponding to $2$ and $1$ must be adjacent in any occurrence. 
In the permutation $\pi = 543162$, the subsequence $546$ is order-isomorphic to $213$, and since $5$ and $4$ are consecutive in $\pi$, it forms a valid occurrence of $\overline{21}3$ in $\pi$. 
In contrast, the subsequence $536$ is also order-isomorphic to $213$, but $5$ and $3$ are not adjacent in $\pi$, so it is not an occurrence of $\overline{21}3$ in $\pi$.

Arrow patterns, introduced by Berman and Tenner~\cite{Berman}, represent a novel generalization of vincular patterns that was motivated by their study of shallow permutations. 
An \emph{arrow pattern} is formally defined as a pair $\alpha = (\nu; H)$, where $\nu = \nu_1\nu_2\cdots \nu_m$ is a string of distinct positive integers (the underlying classical pattern) and $H$ is a (possibly empty) collection of arrows $b_i \to c_i$. 
The distinct integers appearing in either $\nu$ or $H$ collectively form the set $[k]$, and this integer $k$ is called the \emph{size} of the arrow pattern. 
We denote by $\mathcal{A}_k$ the set of all arrow patterns of size $k$.

To define containment of an arrow pattern $\alpha = (\nu; H)$ in a permutation $\pi \in \mathcal{S}_n$, we require two simultaneous conditions:
\begin{enumerate}
	\item there must exist a subset $X = \{x_1 < x_2 < \cdots < x_k\} \subseteq [n]$ and indices $t_1 < t_2 < \cdots < t_m$ such that the subsequence $\pi_{t_1}\pi_{t_2}\cdots\pi_{t_m}$ is precisely $x_{\nu_1}x_{\nu_2}\cdots x_{\nu_m}$ in value; that is, the selected entries form an occurrence of the classical pattern $\nu$ in the one-line notation of $\pi$;
	\item for each arrow $b_i \to c_i$ in $H$, we have $\hat{\pi}(x_{b_i}) = x_{c_i}$; in other words, the cycle structure of $\hat{\pi}$ must map the selected value corresponding to $b_i$ to the selected value corresponding to $c_i$.
\end{enumerate}


For example, consider the arrow pattern $\alpha = (142; 2 \to 3)\in \mathcal{A}_4$. If a permutation $\pi$ contains $\alpha$, then there exists a set $X = \{x_1 < x_2 < x_3 < x_4\}$ such that the subsequence $x_1x_4x_2$ forms an occurrence of the classical pattern $142$ in $\pi$, and moreover $\hat{\pi}(x_2) = x_3$.
Let $\pi = 4167253$. Its left-to-right maxima are $4,6,7$, so
$\hat{\pi} = \theta^{-1}(\pi) = (41)(6)(7253)$.
Now choose $X = \{1,2,5,7\}$, so that $x_1=1$, $x_2=2$, $x_3=5$, and $x_4=7$. The subsequence $172$ of $\pi$ is order-isomorphic to $142$. Furthermore, in the cycle structure we have $\hat{\pi}(2)=5=x_3$. Hence $\pi$ contains the arrow pattern $\alpha$, with $(172;\,2\to 5)$ serving as an occurrence of $\alpha$.
Two arrow patterns $\alpha$ and $\beta$ are said to be \emph{arrow-Wilf-equivalent}, written $\alpha \sim \beta$, if $|\mathcal{S}_n(\alpha)| = |\mathcal{S}_n(\beta)|$ for all $n \ge 1$.

The framework of arrow patterns has proven to be remarkably effective for studying permutation classes that involve both one-line and cycle structure constraints. Berman and Tenner~\cite{Berman} used arrow avoidance to characterize shallow permutations, and Laudone~\cite{Laudone} recently showed that the class of $321$-avoiding cyclic permutations can be characterized entirely by arrow pattern avoidance. 
These results suggest that arrow patterns serve as a bridge between the positional and algebraic aspects of permutations.

Building upon these developments, Archer and Laudone~\cite{Archer} initiated a systematic study of arrow pattern avoidance. They established several structural results and enumerated numerous arrow avoidance classes for patterns of size at most $3$. At the end of their paper, they posed several conjectures regarding the interaction between arrow avoidance and classical pattern avoidance. These conjectures were subsequently resolved by Fu and Yang~\cite{Fu}, who provided complete enumerations for the relevant subclasses.

In this paper, we continue this line of investigation by deriving additional structural properties of arrow patterns and presenting new enumerative results. In particular, we establish several general arrow-Wilf equivalences,
 and provide explicit enumeration formulas for all arrow patterns of the form $(\nu; b\to c)\in \mathcal{A}_3$ where $\nu \in \{31, 23, 32\}$ and $b,c\in[3]$, with the exception of $(32;1\to 3)$. 


The rest of this paper is organized as follows. In Section \ref{sec:general}, we establish the foundational structural results that will be used throughout the paper, including a key lemma that translate arrow patterns into vincular patterns and a series of arrow-Wilf equivalences arising from reversal, complementation, and insertion operations. 
In Section \ref{sec:12-33}, we resolve the two cases left open by Archer and Laudone~\cite{Archer}, namely $(12; 3\to 3)$ and $(21; 3\to 3)$, providing explicit enumeration formulas involving derangement numbers. 
In Sections \ref{sec:31-H}, \ref{sec:23-H}, and \ref{sec:32-H}, we systematically enumerate the remaining arrow patterns of the form $(\nu; b\to c)\in \mathcal{A}_3$ where $\nu \in \{31, 23, 32\}$, connecting our results to Bell numbers, Catalan numbers, and Bessel numbers. 
The results of Sections \ref{sec:31-H}, \ref{sec:23-H}, and \ref{sec:32-H} are summarized in Tables \ref{tab:31-H}, \ref{tab:23-H}, and \ref{tab:32-H}, respectively.
Finally, in Section~\ref{sec:remark}, we conclude with a summary of our results, a discussion of the remaining open case, and an outline of possible directions for future research.

\section{General arrow-Wilf equivalences}\label{sec:general}

In this section, we develop the foundational tools that will be used throughout the remainder of the paper. We first establish a key lemma that translates arrow patterns into vincular patterns under certain conditions. Building on this correspondence, we derive a series of arrow-Wilf equivalences, some of which are analogous to the classical Wilf equivalences arising from reversal and complementation of permutations, while others arise from more subtle operations. Together, these results provide a versatile toolkit for reducing the enumeration of arrow avoidance classes to simpler known cases.

%

We shall make frequent use of the following lemma, which translates certain arrow patterns into vincular patterns.

\begin{lemma}[{\cite[Lemma 2.7]{Archer}}]\label{lem:Archer}
	Let $\nu = \nu_1\nu_2\cdots\nu_{k-1}$ be a word on $[k]$ with distinct integers, and let $r$ be the unique element of $[k]\setminus\{\nu_1,\ldots,\nu_{k-1}\}$. Then:
	\begin{enumerate}
		\item[(i)] if $r > \nu_i$, then $\mathcal{S}_n(\nu; r\to \nu_i) = \mathcal{S}_n(\nu_1\cdots\nu_{i-1}\overline{r\nu_i}\nu_{i+1}\cdots\nu_{k-1})$;
		\item[(ii)] if $r < \nu_i$, then $\mathcal{S}_n(\nu; \nu_i\to r) = \mathcal{S}_n(\nu_1\cdots\nu_{i-1}\overline{\nu_i r}\nu_{i+1}\cdots\nu_{k-1})$.
	\end{enumerate}
\end{lemma}

For a sequence of distinct integers $w = w_1w_2\cdots w_m$, we say that $w_i$ is a \emph{left-to-right maximum} if $w_i > w_j$ for all $j < i$. 
We denote by $\mathrm{Lmax}(w)$ the set of all left-to-right maxima of $w$. 

\begin{lemma}\label{lem:arrow-vincular}
	Let $\nu = \nu_1\nu_2\cdots\nu_{k-1}$ be a word on $[k]$ with distinct integers, and let $r$ be the unique element of $[k]\setminus\{\nu_1,\ldots,\nu_{k-1}\}$. Then:
	\begin{enumerate}
		\item[(i)] if $\nu_i \notin \operatorname{Lmax}(\nu)$, then $\mathcal{S}_n(\nu; r\to \nu_i) = \mathcal{S}_n(\nu_1\cdots\nu_{i-1}\overline{r\nu_i}\nu_{i+1}\cdots\nu_{k-1})$;
		\item[(ii)] if $r < \max(\nu_1,\ldots,\nu_{i-1})$, then $\mathcal{S}_n(\nu; \nu_i\to r) = \mathcal{S}_n(\nu_1\cdots\nu_{i-1}\overline{\nu_i r}\nu_{i+1}\cdots\nu_{k-1})$.
	\end{enumerate}
\end{lemma}

\pf
We prove only (i), as the proof of (ii) is analogous.
To establish the first inclusion, let $\pi \in \mathcal{S}_n(\nu; r \to \nu_i)$. Then there exists a set $X = \{x_1 < x_2 < \cdots < x_k\}$ such that the subsequence $x_{\nu_1}x_{\nu_2}\cdots x_{\nu_{k-1}}$ occurs in $\pi$, and $\hat{\pi}(x_r) = x_{\nu_i}$. Since $\nu_i \notin \mathrm{Lmax}(\nu)$, the element $x_{\nu_i}$ cannot be the first element of its cycle in $\hat{\pi}$. Hence $x_{\nu_i}$ must immediately follow $x_r$ in $\pi$. Therefore $\pi$ contains the vincular pattern $\nu_1\cdots \nu_{i-1}\overline{r\nu_i}\nu_{i+1}\cdots\nu_{k-1}$.

For the reverse inclusion, suppose $\pi \in \mathcal{S}_n(\nu_1\cdots \nu_{i-1}\overline{r\nu_i}\nu_{i+1}\cdots\nu_{k-1})$. Then there exists $X = \{x_1 <x_2 < \cdots < x_k\}$ such that the subsequence
$x_{\nu_1}\cdots x_{\nu_{i-1}}x_r x_{\nu_i}x_{\nu_{i+1}}\cdots x_{\nu_{k-1}}$
occurs in $\pi$, with $x_r$ and $x_{\nu_i}$ adjacent. 
Since $x_r$ appears immediately before $x_{\nu_i}$ in $\pi$, either they belong to the same cycle of $\hat{\pi}$ with $\hat{\pi}(x_r) = x_{\nu_i}$, or $x_{\nu_i}$ is the first element of a new cycle immediately following $x_r$.
The latter case, however, would require $x_{\nu_i}$ to be a left-to-right maximum in $\pi$, which would imply $\nu_i \in \mathrm{Lmax}(\nu)$, contradicting our assumption.
 Hence they must lie in the same cycle and $\hat{\pi}(x_r) = x_{\nu_i}$. Therefore $\pi$ contains the arrow pattern $(\nu; r \to \nu_i)$. This proves the reverse inclusion and completes the proof.
\qed

For a word $\nu = \nu_1\nu_2\cdots\nu_k$, we denote by $\nu^{(i,r)}$ the word obtained from $\nu$ by replacing the entry $\nu_i$ with $r$. That is,
\[
\nu^{(i,r)} = \nu_1\cdots \nu_{i-1}\, r\, \nu_{i+1}\cdots \nu_k.
\]

\begin{proposition}\label{prop:ex}
	Let $\nu = \nu_1\nu_2\cdots\nu_{k-1}$ be a word on $[k]$ with distinct  integers, and let $r$ be the unique element of $[k]\setminus\{\nu_1,\ldots,\nu_{k-1}\}$. 
	If $r > \nu_i$ or $\nu_i \notin \mathrm{Lmax}(\nu)$, then we have 
	\[(\nu; r\to \nu_i) \sim (\nu^{(i,r)}; r\to \nu_i)\sim (\nu_1\cdots \nu_{i-1}r\nu_{i}\nu_{i+1}\cdots\nu_{k-1}; r\to \nu_i).\]
\end{proposition}

\begin{proof}
	We prove only the case $\nu_i \notin \mathrm{Lmax}(\nu)$; the case $r > \nu_i$ is analogous.	
	By Lemma~\ref{lem:arrow-vincular}(i), we have
	\[
	\mathcal{S}_n(\nu; r \to \nu_i) = \mathcal{S}_n(\nu_1\cdots \nu_{i-1}\overline{r\nu_i}\nu_{i+1}\cdots\nu_{k-1}).
	\]
	By Lemma~\ref{lem:arrow-vincular}(ii), we have 
	\[
	\mathcal{S}_n(\nu^{(i,r)}; r \to \nu_i) = \mathcal{S}_n(\nu_1\cdots \nu_{i-1}\overline{r\nu_i}\nu_{i+1}\cdots\nu_{k-1}).
	\]
	Hence $(\nu; r\to \nu_i) \sim (\nu^{(i,r)}; r\to \nu_i)$.
	
	It remains to establish the second equivalence. Let $\mu = \nu_1\cdots \nu_{i-1}r\nu_i\nu_{i+1}\cdots\nu_{k-1}$. By the same argument as in the proof of Lemma~\ref{lem:arrow-vincular}, we obtain
	\[
	\mathcal{S}_n(\mu; r \to \nu_i) = \mathcal{S}_n(\nu_1\cdots \nu_{i-1}\overline{r\nu_i}\nu_{i+1}\cdots\nu_{k-1}).
	\]
	Therefore $(\nu; r\to \nu_i) \sim (\mu; r\to \nu_i)$. Combining the two equivalences yields the desired result.
\end{proof}

We recall some standard notation and operations on permutations. For a permutation $\pi = \pi_1\pi_2\cdots\pi_n \in \mathcal{S}_n$, the \emph{reverse} of $\pi$, denoted $\pi^r$, is the permutation $\pi_n\pi_{n-1}\cdots\pi_1$; the \emph{complement} of $\pi$, denoted $\pi^c$, is the permutation $(n+1-\pi_1)(n+1-\pi_2)\cdots(n+1-\pi_n)$; and the \emph{reverse-complement} of $\pi$, denoted $\pi^{rc}$, is $(\pi^r)^c = (\pi^c)^r$. 
These operations extend naturally to arrow patterns, as defined below.
Let $\alpha = (\nu; H)$ be an arrow pattern in $\mathcal{A}_k$, where $\nu = \nu_1\nu_2\cdots \nu_m$ and $H = \{b_i \to c_i : 1 \leq i \leq t\}$. The \emph{reverse} of $\alpha$, denoted by $\alpha^r$, is the arrow pattern $(\nu^r; H^r)$, where $\nu^r = \nu_m \nu_{m-1} \cdots \nu_1$ and $H^r = \{c_i \to b_i : 1 \leq i \leq t\}$.
The \emph{complement} of $\alpha$, denoted by $\alpha^c$, is the arrow pattern $(\nu^c; H^c)$, where $\nu^c = (k+1-\nu_1)(k+1-\nu_2)\cdots(k+1-\nu_m)$ and $H^c = \{k+1-b_i \to k+1-c_i : 1 \leq i \leq t\}$.
The \emph{reverse-complement} of $\alpha$, denoted by $\alpha^{rc}$, is the arrow pattern obtained by applying reverse and complement successively, i.e., $\alpha^{rc} = (\alpha^r)^c = (\alpha^c)^r$.


\begin{proposition}\label{prop:completment}
	Let $\nu = \nu_1\nu_2\cdots\nu_{k-1}$ be a word on $[k]$ with distinct integers, and let $r$ be the unique element of $[k]\setminus\{\nu_1,\ldots,\nu_{k-1}\}$. 
	For $\alpha = (\nu; r\to \nu_i)$ with 
     $r > \nu_i$ for some $i$,  we have 
     $\alpha \sim  \alpha^{rc}$.
\end{proposition}

\begin{proof}
	Let $\nu' = \nu_1\cdots \nu_{i-1}\overline{r\nu_i}\nu_{i+1}\cdots\nu_{k-1}$. 
	By Lemma~\ref{lem:Archer}(i), we have
	\[
	\mathcal{S}_n(\alpha) = \mathcal{S}_n(\nu').
	\]	
	 Since reverse-complementation preserves avoidance cardinalities, we have $|\mathcal{S}_n(\nu')| = |\mathcal{S}_n((\nu')^{rc})|$. 
	 A direct computation gives
	\[
	(\nu')^{rc} = (k+1-\nu_{k-1})\cdots (k+1-\nu_{i+1})\overline{(k+1-\nu_i)(k+1-r)}(k+1-\nu_{i-1})\cdots(k+1-\nu_1).
	\]
	Applying Lemma~\ref{lem:Archer}(ii)  yields $\mathcal{S}_n((\nu')^{rc}) = \mathcal{S}_n(\alpha^{rc})$. 
	Combining all the above equalities yields the desired result.
\end{proof}

\begin{remark}\label{rem:complement}
	By interchanging the roles of $\alpha$ and $\alpha^{rc}$ in Proposition~\ref{prop:completment}, we obtain the symmetric statement: if $\alpha = (\nu; \nu_i \to r)$ with $\nu_i > r$, then $\alpha \sim \alpha^{rc}$.
\end{remark}

\begin{proposition}\label{prop:reverse}
	Let $\nu = \nu_1\nu_2\cdots\nu_{k-1}$ be a word on $[k]$ with distinct integers, and let $r$ be the unique element of $[k]\setminus\{\nu_1,\ldots,\nu_{k-1}\}$. 
	For $\alpha = (\nu; r\to \nu_i)$ with 
	$r > \nu_i$ for some $i$ and $r<max(\nu_{i+1}, \nu_{i+2}, \ldots, \nu_{k-1})$,  we have 
	$\alpha \sim  \alpha^{r}$.
\end{proposition}

\begin{proof}
	We have
	\begin{align*}
		|\mathcal{S}_n(\alpha)|
		&= |\mathcal{S}_n(\nu_1\cdots \nu_{i-1}\overline{r\nu_i}\nu_{i+1}\cdots\nu_{k-1})|  
		&&(\text{by Lemma~\ref{lem:Archer}(i)}) \\
		&= |\mathcal{S}_n(\nu_{k-1}\cdots \nu_{i+1}\overline{\nu_i r}\nu_{i-1}\cdots\nu_1)|  
		&&(\text{by reversal}) \\
		&= |\mathcal{S}_n(\alpha^r)|  
		&&(\text{by Lemma~\ref{lem:arrow-vincular}(ii)}), 
	\end{align*}
	as desired.
\end{proof}

\begin{proposition}
	Let $\nu = \nu_1\nu_2\cdots\nu_{k-1}$ be a word on $[k]$ with distinct integers, and let $r$ be the unique element of $[k]\setminus\{\nu_1,\ldots,\nu_{k-1}\}$. 
	Suppose that $r < \nu_i$ for some $i$ and $\nu_i\nin \mathrm{Lmax}(\nu)$.
	For $\alpha = (\nu; r\to \nu_i)$,  we have 
	$\alpha \sim  \alpha^r \sim \alpha^{c}$.
\end{proposition}

\begin{proof}
	We have
	\begin{align*}
		|\mathcal{S}_n(\alpha)|
		&= |\mathcal{S}_n(\nu_1\cdots \nu_{i-1}\overline{r\nu_i}\nu_{i+1}\cdots\nu_{k-1})|  
		&&(\text{by Lemma~\ref{lem:arrow-vincular}(i)}) \\
		&= |\mathcal{S}_n(\nu_{k-1}\cdots \nu_{i+1}\overline{\nu_i r}\nu_{i-1}\cdots\nu_1)|  
		&&(\text{by reversal}) \\
		&= |\mathcal{S}_n(\alpha^r)|  
		&&(\text{by Lemma~\ref{lem:Archer}(ii)}) \\ 
		&= |\mathcal{S}_n(\alpha^c)|  
		&&(\text{by Remark~\ref{rem:complement} and the fact $(\alpha^r)^{rc} = \alpha^c)$},
	\end{align*}
	as desired.
\end{proof}

\begin{proposition}\label{prop:}
	Let $\nu = \nu_1\nu_2\cdots \nu_{k-2}$ be a word on $[k]$ with distinct integers, and let $i, j$ be two distinct integers in $[k]\setminus\{\nu_1,\ldots,\nu_{k-2}\}$ such that $\nu_1 > j$. 
	Then
	\[
	\mathcal{S}_n(i\nu;\ i \to j) = \mathcal{S}_n(j\nu;\ i \to j).
	\]
\end{proposition}

\begin{proof}
	Suppose first that $\pi$ contains $(i \nu;\ i \to j)$. Then there exists a set $X = \{x_1 < \cdots < x_k\}$ such that the subsequence $x_i x_{\nu_1} \cdots x_{\nu_{k-2}}$ occurs in $\pi$ and $\hat{\pi}(x_i) = x_j$.
	It follows  that the elements $x_i$ and $x_j$ lie in the same cycle of $\hat{\pi}$. 
	In the one-line notation of $\pi$, this implies that either $x_j$ immediately follows $x_i$, or $x_j$ appears to the left of $x_i$ as the first element of that cycle. 
	In either case,  $(x_j x_{\nu_1} \cdots x_{\nu_{k-2}}; x_i\to x_j)$ forms an occurrence of $(j \nu;\ i \to j)$.
	
Conversely, suppose that $\pi$ contains $(j\nu;\ i \to j)$. Then there exists $X = \{x_1 < \cdots < x_k\}$ such that the subsequence $x_j x_{\nu_1} \cdots x_{\nu_{k-2}}$ occurs in $\pi$ and $\hat{\pi}(x_i) = x_j$. We claim that $x_i$ appears to the left of $x_{\nu_1}$ in $\pi$. 
If not, since $x_j$ appears before $x_{\nu_1}$ in $\pi$ and $\hat{\pi}(x_i)=x_j$, the elements $x_i$, $x_j$ and $x_{\nu_1}$ must belong to the same cycle of $\hat{\pi}$, with $x_j$ as its first element. 
But this would imply $x_j > x_{\nu_1}$, contradicting the fact that $x_j < x_{\nu_1}$ (which follows from $\nu_1 > j$). 
Hence $x_i$ must appear before $x_{\nu_1}$ in $\pi$. 
Therefore $(x_i x_{\nu_1} \cdots x_{\nu_{k-2}}; x_i\to x_j)$ forms an occurrence of $(i\nu;\ i \to j)$.
This completes the proof.
\end{proof}

\begin{proposition}\label{prop:mk}
	Let $m$ be a positive integer, and let $i, j, k \in [m-1]$. 
	Then
	\[
	(mk;\ i\to j) \sim (mk;\ j\to i).
	\]
\end{proposition}

\begin{proof}
	Define $\phi: \mathcal{S}_n \to \mathcal{S}_n$ by
	\[
	\phi(\pi) = \theta\bigl((\hat{\pi})^{-1}\bigr).
	\]
	Since both $\theta$ and inversion are bijections, $\phi$ is a bijection. 
	Moreover, $\phi^2 = \mathrm{id}$, so $\phi$ is an involution.
	The map $\phi$ admits a simple description in terms of the one-line notation of $\pi$. Suppose that the left-to-right maxima of $\pi$ occur at positions $1 = p_1 < p_2 < \cdots < p_t$. Then $\phi(\pi)$ is obtained by keeping these left-to-right maxima fixed in both position and value, and reversing each of the intervening blocks
	\[
	\pi_{p_1+1}\cdots\pi_{p_2-1},\quad
	\pi_{p_2+1}\cdots\pi_{p_3-1},\quad \ldots,\quad
	\pi_{p_t+1}\cdots\pi_n.
	\]
	For example, if $\pi = 312745698$, then $\phi(\pi) = 321765498$.
	
	We claim that $\phi$ restricts to a bijection between $\mathcal{S}_n(mk;\ i\to j)$ and $\mathcal{S}_n(mk;\ j\to i)$. Suppose first that $\pi$ contains $(mk;\ i\to j)$. Then there exists a set $X = \{x_1 < \cdots < x_m\}$ such that $x_m$ precedes $x_k$ in the one-line notation of $\pi$, and $\hat{\pi}(x_i) = x_j$. Among all such occurrences, choose one for which $x_m$ is a left-to-right maximum of $\pi$; this is possible since, if the chosen $x_m$ is not a left-to-right maximum, one may replace it by a left-to-right maximum element preceding it, which does not affect the relative order of the pattern. 
	Thus $x_m$ is fixed by the construction of $\phi$, and consequently it remains before $x_k$ in $\phi(\pi)$.
	Let $\sigma = \phi(\pi)$, so $\hat{\sigma} = (\hat{\pi})^{-1}$. Then
\[
\hat{\sigma}(x_j) = (\hat{\pi})^{-1}(x_j) = x_i.
\]
	Hence $(x_mx_k; x_j\to x_i)$ forms an occurrence of the arrow pattern
	$(mk;\ j\to i)$ in $\phi(\pi)$.
	Conversely, if $\pi$ contains $(mk;\ j\to i)$, then by the same argument, $\phi(\pi)$ contains $(mk;\ i\to j)$. 
	Since $\phi$ is an involution, it establishes a bijection between $\mathcal{S}_n(mk;\ i\to j)$ and $\mathcal{S}_n(mk;\ j\to i)$, as desired.
\end{proof}

\begin{proposition}
	Let $i, j, k$ be three distinct integers in $[3]$. Then
	\[
	(k;\ i\to j) \sim (k;\ j\to i).
	\]
\end{proposition}

\begin{proof}
	This follows directly by applying the map $\phi$ defined in the proof of Proposition~\ref{prop:mk}, which gives a bijection between $\mathcal{S}_n(k;\ i\to j)$ and $\mathcal{S}_n(k;\ j\to i)$.
\end{proof}

\begin{proposition}\label{prop:1-fixed}
	Let $\nu$ be a permutation on the set $\{2,3,\ldots,k\}$. Then
	\[
	\mathcal{S}_n(\nu; 1\to 1) = \mathcal{S}_n(1\nu; 1\to 1).
	\]
\end{proposition}

\begin{proof}
	It is immediate that if a permutation contains $(1\nu; 1\to 1)$, then it also contains $(\nu; 1\to 1)$.
	It remains to prove the reverse implication.
	Suppose $\pi$ contains $(\nu; 1\to 1)$. Then there exists a set $X = \{x_1 < x_2 < \cdots < x_k\}$ such that the subsequence $x_{\nu_1} x_{\nu_2} \cdots x_{\nu_{k-1}}$ occurs in $\pi$ and $\hat{\pi}(x_1) = x_1$. Since $x_1$ is a fixed point of $\hat{\pi}$, it is a left-to-right maximum of $\pi$. Consequently, all elements larger than $x_1$, including $x_{\nu_1}, \ldots, x_{\nu_{k-1}}$, must appear to the right of $x_1$. Therefore, $(x_1x_{\nu_1} x_{\nu_2} \cdots x_{\nu_{k-1}}; x_1\to x_1)$
	forms an occurrence of the pattern $(1\nu; 1\to 1)$ in $\pi$. 
	This proves the reverse implication and completes the proof.
\end{proof}

\begin{proposition}\label{lem:trivial-avoidance}
	Let $\alpha = (\nu; b\to c)\in \mathcal{A}_k$ be an arrow pattern such that $b$ and $c$ both appear in $\nu$. Then $\mathcal{S}_n(\alpha) = \mathcal{S}_n$ if any of the following conditions holds:
	\begin{enumerate}
		\item[(i)] $b > c$, and $c$ does not appear immediately to the right of $b$ in $\nu$;
		\item[(ii)] $b < c$, $b$ appears to the left of $c$ in $\nu$, and $b$ and $c$ are not adjacent in $\nu$;
		\item[(iii)] $b < c$, $b$ appears to the right of $c$ in $\nu$, and there exists an element greater than $c$ preceding $b$ in $\nu$;
		\item[(iv)] $b = c$, and $b \notin \mathrm{Lmax}(\nu)$.
	\end{enumerate}
\end{proposition}

\begin{proof}
	We prove only (i); the remaining cases are analogous and are left to the reader.
	
	Assume $b > c$ and $c$ is not immediately to the right of $b$ in $\nu$. Suppose for contradiction that some permutation $\pi$ contains $(\nu; b\to c)$. Then there exists a set $X = \{x_1 < \cdots < x_k\}$ such that the subsequence $x_{\nu_1}x_{\nu_2}\cdots x_{\nu_k}$ occurs in $\pi$ and $\hat{\pi}(x_b) = x_c$. Since $\hat{\pi}(x_b) = x_c$, the element $x_c$ must appear immediately to the right of $x_b$ in the one-line notation of $\pi$; otherwise, if $x_b$ were the last element of its cycle, then $x_c$ would be the first element of that cycle, implying $x_c > x_b$, contradicting $b > c$. Thus $x_b$ and $x_c$ must be adjacent in $\pi$ with $x_b$ immediately followed by $x_c$. But this would force $b$ and $c$ to be adjacent in $\nu$ with $c$ immediately to the right of $b$, contradicting the assumption. Hence no such occurrence exists, so every permutation avoids $\alpha$.
\end{proof}

We conclude this section with a known result that will be used frequently throughout the remainder of the paper.

\begin{lemma}[{\cite{Claesson}}] \label{lem:vincular}
	For $n \geq 1$, we have
		\[
		|\mathcal{S}_n(\nu)| =
		\begin{cases}
			B_n, & \text{if } \nu \in \{1\overline{23},\  \overline{12}3,\ 1\overline{32}, \ \overline{21}3,\ \overline{23}1, \  3\overline{12}, \ 3\overline{21}, \  \overline{32}1\}, \\[4pt]
			C_n, & \text{if } \nu \in \{\overline{13}2,\ 2\overline{13},\ 2\overline{31},\  \overline{31}2\},
		\end{cases}
		\]
		where $B_n$ denotes the $n$-th Bell number and $C_n$ denotes the $n$-th Catalan number.
\end{lemma}


\section{Patterns $(12; 3\to 3)$ and $(21; 3\to 3)$}\label{sec:12-33}

Archer and Laudone~\cite{Archer} have enumerated $\mathcal{S}_n(\alpha)$ for all arrow patterns $\alpha = (\nu; H)\in \mathcal{A}_k$ with $k \le 3$, $\nu \in \{12, 21, 13\}$, and $|H| = 1$, leaving only the two patterns $(12; 3\to 3)$ and $(21; 3\to 3)$ unresolved. In this section, we complete the enumeration by treating these two remaining cases and deriving explicit formulas for their avoidance counts.

\begin{theorem}\label{thm:12-33}
	For $n\geq 1$, we have 
	\begin{equation*}
		|\mathcal{S}_n(12; 3\to 3)| = d_n + \sum_{m=1}^{n}\sum_{k=0}^{n-m}\binom{n-m}{k}\binom{m+k-1}{n-m} d_k,
	\end{equation*}
	where $d_n$ denotes the $n$-th derangement number.
\end{theorem}

\begin{proof}
	Let $\pi \in \mathcal{S}_n(12; 3\to 3)$. We distinguish two cases according to the fixed points of $\hat{\pi}$.
	
	\textbf{Case 1: $\hat{\pi}$ has no fixed points.} 
	Then $\pi$ automatically avoids $(12;3\to 3)$, since the arrow condition $3\to 3$ requires a fixed point. 
	Hence this case contributes $d_n$ permutations.
	
	\textbf{Case 2: $\hat{\pi}$ has at least one fixed point.} 
	Let $m$ be the largest fixed point of $\hat{\pi}$, where $1\le m\le n$. To avoid $(12;3\to 3)$, the elements $1,2,\ldots,m-1$ must appear in $\pi$ in strictly decreasing order.
	Moreover, since $m$ is the largest fixed point, every element to the left of $m$ is smaller than $m$, and every element to the right of $m$ is not a fixed point of $\hat{\pi}$. Furthermore, if $m$ is not the last element of $\pi$, then the element immediately to its right is larger than $m$.
	
	We now construct $\pi$ as follows. First arrange the elements $m+1,m+2,\ldots,n$ arbitrarily to form a permutation $\sigma$ and place it to the right of $m$. 
	Suppose $\hat{\sigma} = \theta^{-1}(\sigma)$ has exactly $k$ non-fixed points. 
	Then such a $\sigma$ can be chosen in $\binom{n-m}{k} d_k$ ways: choose which $k$ elements are non-fixed and arrange them as a derangement, while the remaining elements are fixed points.
	
	Next, insert the elements $1,2,\ldots,m-1$ into the gaps before $m$ and after each element of $\sigma$. 
	Since $1,2,\ldots,m-1$ must appear in strictly decreasing order, the insertion is completely determined by the number of elements placed in each gap. 
	There are $n-m+1$ available gaps: one before $m$, and one after each element of $\sigma$.
	To preserve the maximality of $m$ as a fixed point, every fixed point of $\hat{\sigma}$ must be followed by at least one inserted element.	
	The number of ways to distribute the $m-1$ decreasing elements into these $n-m+1$ gaps is $\binom{m+k-1}{n-m}$.	
	Therefore, for fixed $m$ and $k$, this case contributes
	\[
	\binom{n-m}{k}\binom{m+k-1}{n-m} d_k
	\]
	permutations. 
	Summing over all admissible $m$ and $k$, together with the derangement contribution from Case 1, yields the desired formula.
\end{proof}

The proof of the following theorem is entirely analogous to that of Theorem~\ref{thm:12-33}, with the sole difference that in Case 2, the elements $1, 2, \ldots, m-1$ must appear in strictly increasing order in $\pi$, rather than decreasing order. We therefore omit the details and leave the verification to the reader.

\begin{theorem}\label{thm:21-33}
	For $n\geq 1$, we have 
	\begin{equation*}
		|\mathcal{S}_n(21; 3\to 3)| = d_n + \sum_{m=1}^{n}\sum_{k=0}^{n-m}\binom{n-m}{k}\binom{m+k-1}{n-m} d_k,
	\end{equation*}
	where $d_n$ denotes the $n$-th derangement number.
\end{theorem}

By Theorem~\ref{thm:12-33} and Theorem~\ref{thm:21-33}, we have 
$(12; 3\to 3) \sim (21; 3\to 3)$.

\section{Permutations that avoid $(31; b\to c)$}\label{sec:31-H}


We now enumerate the arrow patterns of the form $(31; b\to c)\in \mathcal{A}_3$. The enumeration results are summarized in Table \ref{tab:31-H}.

\begin{table}[htbp]
	\centering
	\begin{tabular}{ccccc}
		\toprule
		$\nu$ & $H$ & $a_n = |\mathcal{S}_n(\nu;H)|$ & OEIS & Theorem \\
		\midrule
		$31$ & $1\to 2$ & $B_n$ & A000110 & Theorem \ref{thm:31-12} \\
		$31$ & $2\to 1$ & $B_n$ & A000110 & Theorem \ref{thm:31-21} \\
		$31$ & $3\to 2$ & $B_n$ & A000110 & Theorem \ref{thm:31-32} \\
		$31$ & $2\to 3$ & $B^*_n$ & A006789 & Theorem \ref{thm:31-23} \\
		$31$ & $2\to 2$ & $d_n +  \sum_{m=1}^{n} d_{n-m} a_{m-1}$ & A259870 & Theorem \ref{thm:31-22} \\
		\bottomrule
	\end{tabular}
	\caption{Enumeration of permutations avoiding the arrow pattern $(31;H)\in \mathcal{A}_3$ for various $H$.}
	\label{tab:31-H}
\end{table}

\begin{theorem}\label{thm:31-12}
	For $n\geq 1$, we have 
		$|\mathcal{S}_n(31; 1\to 2)| = B_n$.
\end{theorem}

\begin{proof}
	By Lemma~\ref{lem:arrow-vincular}(ii) and Lemma~\ref{lem:vincular}, we have
	$|\mathcal{S}_n(31; 1\to 2)| = |\mathcal{S}_n(3\overline{12})| = B_n$. 
\end{proof}

\begin{theorem}\label{thm:31-21}
	For $n\geq 1$, we have 
	$|\mathcal{S}_n(31; 2\to 1)| = B_n$.
\end{theorem}

\begin{proof}
	By Lemma~\ref{lem:Archer}(i) and Lemma~\ref{lem:vincular}, we have
	$|\mathcal{S}_n(31; 2\to 1)| = |\mathcal{S}_n(3\overline{21})| = B_n$. 
\end{proof}

\begin{theorem}\label{thm:31-32}
	For $n\geq 1$, we have 
	$|\mathcal{S}_n(31; 3\to 2)| = B_n$.
\end{theorem}

\begin{proof}
	By Lemma~\ref{lem:Archer}(ii) and Lemma~\ref{lem:vincular}, we have
	$|\mathcal{S}_n(31; 3\to 2)| = |\mathcal{S}_n(\overline{32}1)| = B_n$. 
\end{proof}

\begin{theorem}\label{thm:31-23}
	For $n\geq 1$, we have 
	$|\mathcal{S}_n(31; 2\to 3)| = B^*_n$, where $B^*_n$ denotes the $n$-th 
	Bessel number.
\end{theorem}

\begin{proof}
	Let $\pi \in \mathcal{S}_n(31; 2\to 3)$. We first establish a necessary condition on the cycle structure of $\hat{\pi}$.
	
	We claim that every cycle $(c_1, c_2, \ldots, c_k)$ of $\hat{\pi}$ must satisfy $c_1 > c_2 > \cdots > c_k$.
	Suppose otherwise that there exists some $i$ with $c_i < c_{i+1}$. 
	We distinguish two cases.	
	If $i+1 = k$, then $c_{i+1}$ is the last element of the cycle. 
	Since $c_1$ is the maximum of the cycle, we have $c_1 > c_{i+1}$. 
	Then $(c_1 c_i;  c_{i+1}\to c_1)$ would form an occurrence of $(31; 2\to 3)$ in $\pi$, a contradiction.	
	If $i+1 < k$, then either $c_k > c_i$ or $c_k < c_i$. If $c_k > c_i$, then $(c_1c_i; c_k \to c_1)$ is an occurrence of $(31; 2\to 3)$. 
	If $c_k < c_i$, then $(c_{i+1}c_k; c_i \to c_{i+1})$ is an occurrence of $(31; 2\to 3)$. Both cases yield a contradiction. Hence every cycle of $\hat{\pi}$ is strictly decreasing.
	
	Next, we claim that for any two non-singleton cycles of $\hat{\pi}$
	\[
	C = (c_1, c_2, \ldots, c_k) \quad \text{and} \quad D = (d_1, d_2, \ldots, d_\ell),
	\]
	with $c_1 < d_1$, we must have $c_k < d_\ell$. 
	Suppose otherwise that $c_k > d_\ell$. Then 
	$(c_1 d_{\ell}; c_k \to c_1)$
	 forms an occurrence of $(31; 2\to 3)$, a contradiction. 
	 Thus the last elements of the non-singleton cycles of $\hat{\pi}$ are ordered in increasing order.	
	Conversely, it is straightforward to verify that if $\hat{\pi}$ satisfies the two conditions above, then $\pi$ avoids $(31; 2\to 3)$.
	
Such permutations are in bijection with monotone partitions of $[n]$, where the bijection sends each cycle of $\hat{\pi}$ to a block of the partition. Recall that a partition is \emph{monotone} if its non-singleton blocks can be written in increasing order of their least element and increasing order of their greatest element, simultaneously. It is known that monotone partitions are counted by the Bessel numbers $B_n^*$~\cite{Claesson}. This completes the proof.
\end{proof}

\begin{theorem}\label{thm:31-22}
	Let $a_n = |\mathcal{S}_n(31; 2\to 2)|$. 
	Then for $n\geq 1$, we have 
	\begin{equation}\label{equ:31-22}
		a_n = d_n +  \sum_{m=1}^{n} d_{n-m} a_{m-1},
	\end{equation}
	with $a_0 = 1$, where $d_n$ denotes the $n$-th derangement number.
\end{theorem}

\begin{proof}
	Archer and Laudone proved in \cite[Theorem 4.2]{Archer} that the sequence $|\mathcal{S}_n(21; 2\to 2)|$ satisfies the same recurrence as in Eq.~\eqref{equ:31-22}, with the same initial value $a_0 = 1$.
	It therefore suffices to show that $\mathcal{S}_n(31; 2\to 2) = \mathcal{S}_n(21; 2\to 2)$.
	
	If a permutation contains $(31; 2\to 2)$, then the fixed point $2$ must appear before the larger element $3$, since every fixed point is a left-to-right maximum. Hence the permutation also contains $(21; 2\to 2)$.
	Conversely, if a permutation contains $(21; 2\to 2)$, then the fixed point $2$ must be immediately followed by a larger element, which serves as the element $3$ in the pattern $(31; 2\to 2)$. Thus the permutation also contains $(31; 2\to 2)$.	
	Therefore the two avoidance sets are equal, and the desired formula follows.
\end{proof}

\section{Permutations that avoid $(23; b\to c)$}\label{sec:23-H}

We now enumerate the arrow patterns of the form $(23; b\to c)\in \mathcal{A}_3$. The enumeration results are summarized in Table \ref{tab:23-H}.

\begin{table}[htbp]
	\centering
	\begin{tabular}{ccccc}
		\toprule
		$\nu$ & $H$ & $a_n = |\mathcal{S}_n(\nu;H)|$ & OEIS & Theorem \\
		\midrule
		$23$ & $2\to 1$ & $B_n$ & A000110 & Theorem \ref{thm:23-21} \\
		$23$ & $1\to 2$ & $B_n$ & A000110 & Theorem \ref{thm:23-12} \\
		$23$ & $3\to 1$ & $C_n$ & A000108 & Theorem \ref{thm:23-31} \\
		$23$ & $1\to 3$ & $C_n$ & A000108 & Theorem \ref{thm:23-13} \\
		$23$ & $1\to 1$ &  \text{Eq.~\eqref{equ:23-11}}  & New & Theorem \ref{thm:23-11} \\
		\bottomrule
	\end{tabular}
	\caption{Enumeration of permutations avoiding the arrow pattern $(23;H)\in \mathcal{A}_3$ for various $H$.}
	\label{tab:23-H}
\end{table}

\begin{theorem}\label{thm:23-21}
	For $n\geq 1$, we have 
	$|\mathcal{S}_n(23; 2\to 1)| = B_n$.
\end{theorem}

\begin{proof}
	By Lemma~\ref{lem:Archer}(ii) and Lemma~\ref{lem:vincular}, we have
	$|\mathcal{S}_n(23; 2\to 1)| = |\mathcal{S}_n(\overline{21}3)| = B_n$. 
\end{proof}

\begin{theorem}\label{thm:23-12}
	For $n\geq 1$, we have 
	$|\mathcal{S}_n(23; 1\to 2)| = B_n$.
\end{theorem}

\begin{proof}
We claim that $\mathcal{S}_n(23; 1\to 2) = \mathcal{S}_n(13; 1\to 2)$. 	
If a permutation $\pi$ contains $(13;1\to 2)$, then there exist $x_1<x_2<x_3$ such that $x_1$ appears before $x_3$ in $\pi$ and $\hat{\pi}(x_1)=x_2$. 
This implies that either $x_2$ immediately follows $x_1$, or $x_2$ appears before $x_1$ as the first element of that cycle. 
In either case, $x_2$ must appear before $x_3$ in $\pi$.
 Hence $(x_2x_3; x_1\to x_2)$ forms an occurrence of $(23;1\to 2)$ in $\pi$.

Conversely, if $\pi$ contains $(23;1\to 2)$, then there exist $x_1<x_2<x_3$ such that $x_2$ appears before $x_3$ in $\pi$ and $\hat{\pi}(x_1)=x_2$. We must have $x_1$ before $x_3$ in $\pi$; otherwise, $x_1$ would appear after $x_3$, and hence after $x_2$ as well. 
But since $\hat{\pi}(x_1)=x_2$, the element $x_2$ would be the first element of a cycle containing $x_3$, forcing $x_2 > x_3$, a contradiction. 
Hence $x_1$ appears before $x_3$ in $\pi$, and $(x_1x_3; x_1\to x_2)$ forms an occurrence of $(13;1\to 2)$ in $\pi$.
	
Therefore the two avoidance sets coincide. By Theorem~5.1 of~\cite{Archer}, $|\mathcal{S}_n(13; 1\to 2)| = B_n$, and hence $|\mathcal{S}_n(23; 1\to 2)| = B_n$.
\end{proof}

\begin{theorem}\label{thm:23-31}
	For $n\geq 1$, we have 
	$|\mathcal{S}_n(23; 3\to 1)| = C_n$.
\end{theorem}

\begin{proof}
	By Lemma~\ref{lem:Archer}(ii) and Lemma~\ref{lem:vincular}, we have
	$|\mathcal{S}_n(23; 3\to 1)| = |\mathcal{S}_n(2\overline{31})| = C_n$. 
\end{proof}

\begin{theorem}\label{thm:23-13}
	For $n\geq 1$, we have 
	$|\mathcal{S}_n(23; 1\to 3)| = C_n$.
\end{theorem}

\begin{proof}
	Let $\pi \in \mathcal{S}_n(23; 1\to 3)$. Consider two cycles of $\hat{\pi}$,
	\[
	C = (c_1, c_2, \ldots, c_k) \quad \text{and} \quad D = (d_1, d_2, \ldots, d_\ell),
	\]
	with $c_1 < d_1$. We claim that $c_1 < d_j$ for all $1 \le j \le \ell$. Suppose otherwise, and let $d_j$ be the last element of $D$ that is smaller than $c_1$. If $j = \ell$, then $(c_1d_1; d_{\ell}\to d_1)$ forms an occurrence of $(23; 1\to 3)$ in $\pi$, a contradiction. 
	If $j < \ell$, then $(c_1d_{j+1}; d_j \to d_{j+1})$ forms an occurrence of $(23; 1\to 3)$ in $\pi$, again a contradiction. 
	Hence the cycles of $\hat{\pi}$ are totally ordered by their elements: every element of a preceding cycle is smaller than every element of any subsequent cycle.
	
	Now let the first cycle of $\hat{\pi}$ have length $s$, where $1 \le s \le n$. By the ordering property above, its elements must form the interval $[s]$, and the first element of the cycle, being its maximum, must be $s$. In the one-line notation, this means $\pi_1 = s$. Write
	\[
	\pi = s\, \sigma\, \tau,
	\]
	where $\sigma = \pi_2\cdots\pi_s$ and $\tau = \pi_{s+1}\cdots\pi_n$. It is straightforward to verify that $\pi$ avoids $(23; 1\to 3)$ if and only if $\sigma$ avoids the vincular pattern $2\overline{13}$ and $\tau$ avoids $(23; 1\to 3)$. By Lemma~\ref{lem:vincular}, the number of permutations of length $s-1$ avoiding $2\overline{13}$ is the Catalan number $C_{s-1}$. Let $a_n = |\mathcal{S}_n(23; 1\to 3)|$. Then the above decomposition yields the recurrence
	\[
	a_n = \sum_{s=1}^{n} C_{s-1} a_{n-s},
	\]
	with $a_0 = 1$. This is precisely the recurrence satisfied by the Catalan numbers $C_n$ (with $C_0 = 1$). Therefore $a_n = C_n$ for all $n \ge 1$.
\end{proof}

\begin{theorem}\label{thm:23-11}
	For $n\ge 1$, we have
	\begin{equation}\label{equ:23-11}
		|\mathcal{S}_n(23; 1\to 1)|
		= d_n + d_{n-1}
		+ \sum_{m=1}^{n-1}\sum_{r=0}^{m-1}
		\binom{m-1}{r}\binom{n-m+r-1}{r} r! \cdot d_{m-1-r},
	\end{equation}
	where $d_n$ denotes the $n$-th derangement number.
\end{theorem}

\begin{proof}
	Let $\pi \in \mathcal{S}_n(23; 1\to 1)$. We classify permutations according to the fixed points of $\hat{\pi}$.
	
	\textbf{Case 1: $\hat{\pi}$ has no fixed points.} 
Then $\pi$ automatically avoids $(23;1\to 1)$.
The number of such permutations is exactly the derangement number $d_n$.
	
	\textbf{Case 2: The smallest fixed point of $\hat{\pi}$ is $n$.}
	In this case, $n$ is the only fixed point. Deleting $n$ leaves a derangement on the remaining $n-1$ elements, contributing $d_{n-1}$ permutations.
	
	\textbf{Case 3: The smallest fixed point of $\hat{\pi}$ is $m$, where $1 \le m \le n-1$.}
	Since $m$ is a fixed point of $\hat{\pi}$, it is a left-to-right maximum of $\pi$. Consequently, every element larger than $m$ must appear to its right. Furthermore, these elements $m+1, m+2, \ldots, n$ must occur in strictly decreasing order. 
	Indeed, if two elements $a < b$ both greater than $m$ appeared with $a$ before $b$ to the right of $m$, then $(ab; m\to m)$ would form an occurrence of $(23;1\to 1)$ in $\pi$, a contradiction. 
	
	Now suppose that among the elements $\{1,2,\ldots,m-1\}$, exactly $r$ of them are inserted into the decreasing sequence $n, n-1, \ldots, m+1$ on the right of $m$. There are $\binom{m-1}{r}$ ways to choose these $r$ elements. They can be distributed among the $n-m$ gaps after each element of the decreasing sequence (excluding the gap immediately before $n$) and then permuted arbitrarily, giving
	\[
	\binom{n-m+r-1}{r} r!
	\]
	possibilities. The remaining $m-1-r$ elements, which lie to the left of $m$, must themselves form a derangement, as $m$ is the smallest fixed point of $\hat{\pi}$; they are counted by $d_{m-1-r}$.
	
	Summing over all admissible $m$ and $r$ yields the double sum in the theorem. 
	Combining the three cases completes the proof.
\end{proof}

\section{Permutations that avoid $(32; b\to c)$}\label{sec:32-H}

We now enumerate the arrow patterns of the form $(32; b\to c)\in \mathcal{A}_3$. The enumeration results are summarized in Table \ref{tab:32-H}.

\begin{table}[htbp]
	\centering
	\begin{tabular}{ccccc}
		\toprule
		$\nu$ & $H$ & $a_n = |\mathcal{S}_n(\nu;H)|$ & OEIS & Theorem \\
		\midrule
		$32$ & $1\to 2$ & $B_n$ & A000110 & Theorem \ref{thm:32-12} \\
		$32$ & $2\to 1$ & $B_n$ & A000110 & Theorem \ref{thm:32-21} \\
		$32$ & $3\to 1$ & $C_n$ & A000108 & Theorem \ref{thm:32-31} \\
		$32$ & $1\to 1$ &  \text{Eq.~\eqref{equ:32-11}}  & New & Theorem \ref{thm:32-11} \\
		\bottomrule
	\end{tabular}
	\caption{Enumeration of permutations avoiding the arrow pattern $(32;H)\in \mathcal{A}_3$ for various $H$.}
	\label{tab:32-H}
\end{table}

\begin{theorem}\label{thm:32-12}
	For $n\geq 1$, we have 
	$|\mathcal{S}_n(32; 1\to 2)| = B_n$.
\end{theorem}

\begin{proof}
	By Lemma~\ref{lem:arrow-vincular}(i) and Lemma~\ref{lem:vincular}, we have
	$|\mathcal{S}_n(32; 1\to 2)| = |\mathcal{S}_n(3\overline{12})| = B_n$. 
\end{proof}

\begin{theorem}\label{thm:32-21}
	For $n\geq 1$, we have 
	$|\mathcal{S}_n(32; 2\to 1)| = B_n$.
\end{theorem}

\begin{proof}
	By Lemma~\ref{lem:Archer}(ii) and Lemma~\ref{lem:vincular}, we have
	$|\mathcal{S}_n(32; 2\to 1)| = |\mathcal{S}_n(3\overline{21})| = B_n$. 
\end{proof}

\begin{theorem}\label{thm:32-31}
	For $n\geq 1$, we have 
	$|\mathcal{S}_n(32; 3\to 1)| = C_n$.
\end{theorem}

\begin{proof}
	By Lemma~\ref{lem:Archer}(ii) and Lemma~\ref{lem:vincular}, we have
	$|\mathcal{S}_n(32; 3\to 1)| = |\mathcal{S}_n(\overline{31}2)| = C_n$. 
\end{proof}

The proof of the following theorem is completely analogous to that of Theorem~\ref{thm:23-11}, the only difference being that in Case 3, the elements $m+1, m+2, \ldots, n$ must appear in strictly increasing order to the right of $m$, rather than decreasing order. We therefore omit the details and leave the verification to the reader.

\begin{theorem}\label{thm:32-11}
	For $n\ge 1$, we have
	\begin{equation}\label{equ:32-11}
		|\mathcal{S}_n(32; 1\to 1)|
		= d_n + d_{n-1}
		+ \sum_{m=1}^{n-1}\sum_{r=0}^{m-1}
		\binom{m-1}{r}\binom{n-m+r-1}{r} r! \cdot d_{m-1-r},
	\end{equation}
	where $d_n$ denotes the $n$-th derangement number.
\end{theorem}

Combining Theorem~\ref{thm:23-11} and Theorem~\ref{thm:32-11}, we have $(23; 1\to 1) \sim (32; 1\to 1)$.

%

\section{Final remarks and future directions}\label{sec:remark}

In this paper, we have resolved the two cases left open by Archer and Laudone~\cite{Archer}, namely $(12;3\to 3)$ and $(21;3\to 3)$, and have enumerated the remaining avoidance classes for arrow patterns $(\nu; b\to c)\in \mathcal{A}_3$ with $\nu \in \{31, 23, 32\}$ and $b, c\in [3]$. Together with the results of Archer and Laudone~\cite{Archer}, the only unresolved case for $|\nu| \le 2$ is now $(32; 1\to 3)$.

We have also established several arrow-Wilf equivalences, many of which arise from reversal, complementation, and insertion operations. However, a number of these equivalences were obtained through ad hoc arguments, suggesting that a more unified framework for arrow-Wilf equivalence may exist. We leave this as an open problem for future investigation.

We also record an intriguing observation. We have computed the avoidance counts for the arrow patterns $(12;3\to 3)$ and $(23;1\to 1)$ and found that the resulting sequences agree for all values of $n$ tested. This suggests that these two patterns may be arrow-Wilf-equivalent. However, the explicit formulas we obtained for the two sequences are quite different in form, and we have not been able to establish the equivalence by a direct bijection or by any other method. We leave this as an open problem.

In a separate paper, we will extend this study to arrow patterns $(\nu; b\to c)$ where $\nu \in \mathcal{S}_3$. The enumeration in that setting is considerably more intricate, involving complex recurrences and bijections, and yielding connections to a wider range of combinatorial sequences. 


\section*{Acknowledgments}
	The work  was supported by
	the National Natural Science Foundation of China (11801378).
	

\end{document}